\documentclass[a4paper,12pt]{elsarticle}

\usepackage{amsthm,amsmath,amssymb,url,color,array,stmaryrd}

\newtheorem{thm}{Theorem}[section]
\newtheorem{lem}[thm]{Lemma}
\newtheorem{cor}[thm]{Corollary}

\theoremstyle{definition}

\newtheorem{pr}[thm]{Problem}

\theoremstyle{remark}
\newtheorem{rem}[thm]{Remark}

\numberwithin{equation}{section}

\def\printdate{\hfill\small\it\today}
\makeatletter
\def\ps@pprintTitle{%
  \let\@oddhead\@empty
  \let\@evenhead\@empty
  \let\@oddfoot\printdate
  \let\@evenfoot\@oddfoot
}
\makeatother

\def\st{:\,}

\def\le{\leqslant}
\def\ge{\geqslant}

\def\2{^{2}}
\def\p{^{p}}

\begin{document}

\begin{frontmatter}

\title{Totally Silver Graphs}


\author[mg]{M. Ghebleh}
\address[mg]{Department of Mathematics,
             Faculty of Science,
             Kuwait University,
             State of Kuwait}
\ead[mg]{mamad@sci.kuniv.edu.kw}

\author[ebad]{E. S. Mahmoodian}
\address[ebad]{Department of Mathematical Sciences,
             Sharif University of Technology,
             Tehran, Iran}
\ead[ebad]{emahmood@sharif.ir}

\begin{abstract}
A {\em totally silver coloring} of a graph $G$ is a $k$--coloring of $G$ such that
for every vertex $v\in V(G)$, each color appears exactly once on $N[v]$, the
closed neighborhood of~$v$. A {\em totally silver graph} is a graph which admits
a totally silver coloring. 
Totally silver coloring are directly related to other areas of graph theory such as distance coloring and domination.
In this work, we present several constructive characterizations of
totally silver graphs and bipartite totally silver graphs.
We give several infinite families of totally silver graphs.
We also give cubic totally silver graphs of girth up to~$10$.
\end{abstract}

\end{frontmatter}


\section{Introduction}
\label{sec:intro}


All graphs in this paper are finite and simple. 
A {\em totally silver coloring} of a graph $G$ is a $k$--coloring of $G$ such that
for every vertex $v\in V(G)$, each color appears exactly once on $N[v]$, the
closed neighborhood of~$v$. A {\em totally silver graph} is a graph which admits
a totally silver coloring. It is immediate from this definition that every graph
admitting a totally silver coloring with $k$ colors is $(k-1)$--regular.
Totally silver colorings are introduced in~\cite{silvercubes}. The term totally
silver is inspired by an International Mathematical Olympiad problem on
silver matrices~\cite{IMO97}.

Totally silver colorings appear, under the name {\em perfect colorings}, in the study of file segmentations in a network in~\cite{perfectcol}.
Totally silver colorings are also studied in~\cite{strongcol} where they are called {\em strong colorings}.
Totally silver coloring are directly translated to problems in other
areas of graph theory. For instance, it is proved in~\cite{silvercubes} that a
$\Delta$--regular graph is totally silver, if an only if it is domatically
full. That is, if and only its vertices can be partitioned into $\Delta+1$
dominating sets.

Another area of graph theory related to totally silver colorings is  distance coloring.
The $p$th power of a graph $G$, denoted by $G\p$, is the graph with vertex set $V(G)$
in which two distinct vertices are adjacent, if and only if they are at distance at most~$p$
in~$G$. A $\Delta$--regular graph $G$ is totally silver, if and only if $G\2$ is
$\Delta+1$--colorable.
Colorings of
$G\p$ are studied in the literature as {\em distance--$p$ colorings}.
In particular, upper bounds on the number of colors required by a 
distance--$p$ coloring have been studied. If $G$ is a graph with maximum degree $\Delta$, then $G\2$ is $\Delta^2+1$--colorable by Brooks' theorem.
On the other hand, the closed neighborhood of a vertex of degree $\Delta$
induces a $\Delta+1$--clique in~$G\2$. So we have
\begin{equation}\Delta+1\le\chi(G\2)\le\Delta^2+1.\label{eq:bound}\end{equation}
By Brook's theorem, the upper bound above is tight only when $G\2$ is a
$\Delta^2+1$--clique. It is easily seen that this happens precisely when $G$ is a Moore graph with diameter~$2$. Thus for every graph other than the $5$--cycle,
the Petersen graph, the Hoffman--Singleton graph, and the Moore graph with degree
$57$ and diameter $2$ (if it exists), the above upper bound is improved to~$\Delta^2$.
For graphs with maximum degree $3$, this bound is further improved in the following theorem of~\cite{cranston}.

\begin{thm}{\em\cite{cranston}}
Let $G$ be a subcubic graph other than the Petersen graph. Then $G\2$ is
$8$--choosable, that is for every assignment of lists of size at least $8$
to the vertices of $G$, there is a proper coloring of $G\2$ such that
the color of each vertex is chosen from its list.
\end{thm}

Totally silver graphs are regular graphs for which the lower bound of equation~\ref{eq:bound} is tight. We study such graphs in this paper.
In Section~\ref{sec:char} we present several constructive characterizations of
totally silver graphs, as well as a similar characterization of bipartite totally silver graphs. In Section~\ref{sec:triv} we give nontriviality conditions for cubic totally silver graphs. In Section~\ref{sec:eg} we present several families of cubic totally silver graphs. In Section~\ref{sec:gir} some cubic totally silver graphs of high girth are presented.


\section{Properties and Characterization of totally silver graphs}
\label{sec:char}

Let $c$ be a totally silver coloring of an $r$--regular graph~$G$. 
We define a {\em colored $2$--switch} in $G$ to be the following operation.
Let $u_1v_1$ and $u_2v_2$
be edges of $G$ such that $c(u_1)=c(u_2)$ and $c(v_1)=c(v_2)$. Let $G'$ be
obtained from $G$ by deleting the edges $u_1v_1$ and $u_2v_2$, and adding the
edges $u_1v_2$ and $u_2v_1$.
Since a colored $2$--switch does not affect the colors appearing on each closed
neighborhood, $c$ is also a totally silver coloring of~$G'$.
The following theorem gives a classification of $r$--regular totally silver graphs.

\begin{thm}{\em\cite{perfectcol}}
An $r$--regular graph is totally silver if and only if it can be obtained by
a sequence of colored $2$--switches from a disjoint union of $r+1$--cliques.
\label{thm:switch}
\end{thm}

The following is immediate from this theorem.

\begin{cor}
In a totally silver coloring of an $r$--regular graph $G$, all color classes
have equal size. In particular, $|V(G)|$ is a multiple of~$r+1$.
\label{cor:colorclasses}
\end{cor}

The following theorem gives a necessary condition for a graph being totally silver.

\begin{thm}{\em\cite{strongcol}}
Let $r$ be odd. Then every $r$--regular totally silver graph is $r$--edge colorable.
\label{thm:edgecol}
\end{thm}

In the remainder of this section we focus on bipartite totally silver graphs
and give a characterization similar to that of Theorem~\ref{thm:switch} for such graphs.

\begin{lem}
Let $G$ be an $r$--regular bipartite totally silver graph with a bipartition $(X,Y)$,
and let $c$ be a totally silver coloring of $G$ with color classes $C_0,C_1,\ldots,C_r$.
Then $|C_j\cap X|=|C_j\cap Y|$ and this value is independent of~$j$.
In other words, every color classes meets every partite set at the same number of
vertices.
\label{lem:bip}
\end{lem}

\begin{proof}
Let $0\le j\le r$. Since every $x\in X\cap C_j$ has $r$ neighbors in $Y$,
and since every $y\in Y\setminus C_j$ is adjacent to exactly one $x\in X\cap C_j$,
we have \[r|X\cap C_j|=|Y\setminus C_j|=|Y|-|Y\cap C_j|.\]
Rearranging the terms we obtain
\[(r-1)|X\cap C_j|=|Y|-|X\cap C_j|-|Y\cap C_j|=|Y|-|C_j|.\]
The right-hand side of this equation is independent of $j$ since by
Corollary~\ref{cor:colorclasses} we have $|C_0|=|C_1|=\cdots=|C_r|$.
Thus $|X\cap C_j|=|X|/(r+1)$, and similarly $|Y\cap C_j|=|Y|/(r+1)$.
On the other hand, since $G$ is a regular graph, we have $|X|=|Y|$.
Therefore, $|X\cap C_j|=|Y\cap C_j|=|V(G)|/(2r+2)$.
\end{proof}

\begin{cor}
The order of every $r$--regular bipartite totally silver graph is a multiple of~$2r+2$.
\label{cor:2r+2}
\end{cor}

Let $G$ be a bipartite $r$--regular totally silver graph with a bipartition $(X,Y)$,
and let $c$ be a totally silver coloring of~$G$.
The colored $2$--switch introduced above does not preserve the bipartite property
of a graph, since the new edges added may create an odd cycle.
We may amend this operation to resolve this issue by adding the condition that
a colored $2$--switch on the edges $u_1v_1$ and $u_2v_2$ with $c(u_1)=c(u_2)$ and
$c(v_1)=c(v_2)$ is permitted, only when
$u_1$ and $u_2$ belong to the same partite set of $G$. We may refer to this
operation as the {\em bipartite colored $2$--switch}. For every positive integer
$r$ we define a graph $B_r$ to have vertex set
\[V(B_r)=\{x_0,x_1,\ldots,x_r\}\cup\{y_0,y_1,\ldots,y_r\},\]
and edge set
\[E(B_r)=\{x_i y_j\st 0\le i,j\le r\text{ and }i\not=j\}.\]
Then $B_r$ is $r$--regular and bipartite, and $c:V(B_r)\to\{0,1,\ldots,r\}$
defined by $c(x_i)=c(y_i)=i$ for $i=0,1,\ldots,r$ is a totally silver coloring
of~$B_r$. 
Indeed $B_r$ is obtained from a complete bipartite graph
$K_{r+1,r+1}$ by deleting the edges of a perfect matching. The graph $B_2$ is
isomorphic to $C_6$ and $B_3$ is isomorphic to the hypercube graph~$Q_3$.
It is immediate from Lemma~\ref{lem:bip} that $B_r$ is the unique $r$--regular bipartite
totally silver graph of order~$2r+2$.
The following theorem gives a characterization of bipartite totally silver graphs.

\begin{thm}
An $r$--regular bipartite graph is totally silver if and only if it can be obtained
by a sequence of bipartite colored $2$--switches from a disjoint union of copies
of~$B_r$.
\label{thm:bipchar}
\end{thm}

\begin{proof}
In this proof, a {\em switch} refers to a bipartite colored $2$--switch as defined
above.
It is immediate from the definition of a switch that if $G$ is obtained from an
$r$--regular bipartite totally silver graph by means of repeated applications of
switches, then $G$ is $r$--regular, bipartite, and totally silver.

Let $G$ be an $r$--regular bipartite totally silver graph with a bipartition $(X,Y)$.
Let $B$ be a disjoint union of copies of $B_r$ which has the same order as~$G$.
Since a sequence of switches can be reversed in an obvious way, it suffices to
prove that $B$ is obtained from $G$ via a sequence of switches.
We proceed by induction on $|V(G|$. If $|V(G)|=2r+2$, then $G$ is isomorphic to
$H=B_r$ since $B_r$ is the only $r$--regular bipartite totally silver graph of
order~$2r+2$. Let $|V(G)|>2r+2$, and let $c:V(G)\to\{0,1,\ldots,r\}$ be a totally
silver coloring of~$G$.
Let $U=\{u_0,u_1,\ldots,u_r\}\subset X$ and $V=\{v_0,v_1,\ldots,v_r\}\subset Y$
such that $c(u_i)=c(v_i)=i$ for all $i=0,1,\ldots,r$. Such sets $U$ and $V$
exist by Lemma~\ref{lem:bip}. Let $0\le i,j\le r$ with $i\not=j$.
Then there are unique vertices $y\in N(u_i)$ and $x\in N(v_j)$
such that $c(y)=j$ and $c(x)=i$. If $y=v_j$ (and hence $x=u_i$), we do nothing.
Otherwise, we switch the edges $u_i y$ and $v_j x$, which creates the
edge~$u_iv_j$. Performing this operation for all pairs of colors $i$ and $j$,
we obtain a graph $G'$ in which $U\cup V$ induces a subgraph isomorphic to~$B_r$.
Since $G'$ is $r$--regular, it is a disjoint union of this subgraph and a
graph~$H$. By the induction hypothesis, $H$ is obtained by a sequence of switches
from a disjoint union of copies of~$B_r$. The sequence of switches introduced above,
followed by this  latter sequence, gives a sequence of switches which transforms
$G$ to~$B$.
\end{proof}

We conclude this section by some reformulations of the results of Theorem~\ref{thm:switch} and Theorem~\ref{thm:bipchar}.
Theorem~\ref{thm:switch} gives way to the following characterization of $r$--regular
totally silver graphs.

\begin{cor}
A graph $G$ is totally silver if and only if $V(G)$ can be partitioned to
a union of $r+1$ disjoint independent sets $C_0,C_1,\ldots,C_r$ such that
for all distinct $i,j\in\{0,1,\ldots,r\}$, the subgraph induced by $G$
on $C_i\cup C_j$ is $1$--regular.
\label{cor:construction}
\end{cor}

This characterization can be viewed as a way of constructing an $r+1$--regular totally silver graph from any existing $r$--regular totally silver graph.

\begin{cor}
Let $G$ be an $r$--regular  totally silver graph and let $C_0,C_1,\ldots,C_r$ be
as in Corollary~\ref{cor:construction}. Let $C_{r+1}$ be a set disjoint from $V(G)$ such that $|C_{r+1}|=|C_0|$. Then an $r+1$--regular totally silver graph $H$ can be constructed from $G$ by letting $V(H)=V(G)\cup C_{r+1}$, and adding a perfect matching between $C_{r+1}$ and $C_i$ for any $i\in\{0,1,\ldots,r\}$. Moreover, any $r+1$--regular totally silver graph can be constructed from a suitable $r$--regular totally silver graph in this way.
\label{cor:augment}
\end{cor}

A similar characterization of bipartite totally silver graphs can be formulated
as a corollary of Theorem~\ref{thm:bipchar} as follows.

\begin{cor}
A bipartite graph $G$ with bipartition $(X,Y)$ is totally silver,
if and only if $X=U_0\cup U_1\cup\cdots\cup U_r$ and $Y=V_0\cup V_1\cup\cdots\cup V_r$
where the sets $U_i$ and $V_j$ are pairwise disjoint, such that
for any distinct $i,j\in\{0,1,\ldots,r\}$, the subgraph induced by $G$
on $U_i\cup V_j$ is $1$--regular.
\label{cor:bipconst}
\end{cor}

A construction similar to that of Corollary~\ref{cor:augment} may be derived easily from Corollary~\ref{cor:bipconst}


\section{Nontriviality conditions for cubic totally silver graphs}
\label{sec:triv}

If $G$ is a totally silver graph, then every connected component of $G$
is totally silver. 
On the other hand, by Theorem~\ref{thm:edgecol}, every cubic totally silver
graph is bridgeless.
Therefore, every cubic totally silver graph is $2$--connected.
We begin this section by proving that although a $2$--edge cut may be present
in a cubic totally silver graph, its existence leads to a trivial situation,
that is a reduction to smaller cubic totally silver graphs.

\begin{lem}
Every cubic totally silver graph with a $2$--edge cut is obtained from
a disjoint union of two cubic totally silver graphs via a colored $2$--switch.
\end{lem}

\begin{proof}
Let $(S,T)$ be an edge cut in $G$ with $\delta(S)=\{x y,x' y'\}$ and $x,x'\in S$.
Since $G$ is bridgeless, we have $x\not=x'$ and $y\not=y'$.
Let $c$ be a totally silver coloring of $G$ with $c(x)=1$ and $c(y)=2$,
and assume $c(x')\not=4$.
Let $n$ be the number of vertices $v\in S$ with $c(v)=4$.
For $i=1,2,3$, each $v\in S$ with $c(v)=4$ has exactly one neighbor $w$
with $c(w)=i$.
On the other hand, each $w\in S$ with $c(w)=i$ has exactly one neighbor
$v$ with $c(v)=4$.
Thus $|c^{-1}(i)\cap S|=n$. A similar counting argument with
the color $1$ in place of $4$ gives $c(x')=2$ and $c(y')=1$. 
Note that since $c(y)=c(x')$ and $x y$ is an edge, $xx'$ is not an edge.
Thus one may apply a colored $2$--switch on the edges $x y$ and $x' y'$
to obtain two cubic totally silver graphs on vertex sets $S$ and~$T$.
\end{proof}

By the above lemma, a nontrivial cubic totally silver graph is $3$--connected.
We proceed by giving similar reductions for cubic totally silver graphs
with girth less than~$6$.
Note that since the square of a $5$--cycle and the square of $K_{2,3}$
are both isomorphic to a $5$--clique, a cubic totally silver graph does
not contain $C_5$ or $K_{2,3}$ as a subgraph.

\begin{lem}
Let $G$ be a cubic totally silver graph which contains a $3$--cycle.
If $|V(G)|>4$, then $G$ can be reduced to a cubic totally silver graph
on $|V(G)|-4$ vertices.
\end{lem}

\begin{proof}
Let $G$ be a connected cubic totally silver graph and let $c$ be a totally silver
coloring of~$G$. Let $T=\{x_1,x_2,x_3\}$ induce a triangle in~$G$.
We may assume that $c(x_i)=i$ for $i=1,2,3$.
If there is $y\not\in T$ with two neighbors in~$T$,
we construct a graph $G'$ by removing $x_1,x_2,x_3,y$ and all their edges from $G$
and then connecting the two resulting degree $2$ vertices by an edge.
Then the restriction of $c$ to $V(G')$ is a totally silver coloring of~$G'$.
Note that the degree $2$ vertices in this construction are distinct since
otherwise $G$ has a $5$--cycle.
Suppose that each  $x_i$ has a distinct neighbor $y_i$ outside of~$T$.
Let $z_1$ and $z_2$ be the two neighbors of $y_3$ other than $x_3$.
Then $z_1,z_2\not\in T$, $c(y_1)=c(y_2)=c(y_3)=4$ and $\{c(z_1),c(z_2)\}=\{1,2\}$.
By possibly renaming the vertices $z_1$ and $z_2$, we may assume that
$c(z_1)=1$ and $c(z_2)=2$. Let $G''$ be obtained from $G$ by deleting
the vertices $x_1,x_2,x_3,y_3$ and all their edges, and adding the edges
$y_1z_1$ and $y_2z_2$.
The the restriction of $c$ to $V(G'')$ is a totally silver coloring of~$G''$.
\end{proof}

\begin{lem}
Let $G$ be a triangle-free cubic totally silver graph which contains a $4$--cycle.
If $|V(G)|>4$, then $G$ can be reduced to a cubic totally silver graph
on $|V(G)|-4$ vertices.
\label{lem:reduce4cycle}
\end{lem}

\begin{proof}
Let $G$ be a connected triangle-free cubic totally silver graph and let
$c$ be a totally silver coloring of~$G$.
Let $C=x_1x_2x_3x_4x_1$ be a $4$--cycle in~$G$.
Each $x_i$ has a distinct neighbor $y_i\not\in C$ since $G$ is triangle-
and $K_{2,3}$--free.
By possibly renaming the colors, we may assume that $c(x_i)=i$ for $i=1,2,3,4$.
Then $c(y_1)=3$, $c(y_2)=4$, $c(y_3)=1$, and $c(y_4)2$.
Let $G'$ be obtained from $G$ by deleting the vertices $x_1,x_2,x_3,x_4$ and
all their edges, then adding the edges $y_1y_3$ and $y_2y_4$.
Then the restriction of $c$ to $V(G')$ is a totally silver coloring of~$G'$.
\end{proof}

In light of the reductions of this section, we consider a cubic totally silver
graph {\em nontrivial}, if it is $3$--connected and it has girth at least~$6$.


\section{Constructions of cubic totally silver graphs}
\label{sec:eg}

In this section we present several infinite families of nontrivial cubic totally silver
graphs. In particular, we give nontrivial cubic totally silver graphs of any
order which is a multiple of~$4$. Note that the $(3,6)$--cage (the smallest cubic
graph with girth at least $6$) has order~$14$. Therefore, no nontrivial cubic
totally silver graphs of orders $4$, $8$, or $12$ exist.

As our first example, we consider the family of
generalized Petersen graphs $P(n,d)$ where $n$ and $d$ are positive integers with
$n\ge2d+1$. The graph $P(n,d)$ consists of an $n$--cycle $x_1x_2\cdots x_n x_1$,
along with $n$ other vertices $y_1,y_2,\ldots,y_n$ where each $y_i$ is adjacent to
$x_i$, $y_{i-d}$ and $y_{i+d}$ (all subscripts are reduced modulo $n$). In
particular, $P(4,1)$ is the hypercube graph $Q_3$, $P(5,2)$ is the Petersen graph,
and $P(n,1)$ is the prism $C_n\boxempty K_2$, where $\boxempty$ denotes the Cartesian
product.
The following is the main result of~\cite{EbrJahMah}.

\begin{thm}{\em\cite{EbrJahMah}}
The generalized Petersen graph $P(n,d)$ is totally silver
if and only if $4|n$ and $d$ is odd.
\end{thm}

The above theorem gives examples of cubic totally silver graphs whose order is divisible
by~$8$. While these graphs are all $3$--connected and bipartite (thus triangle-free),
not all of them are $C_4$--free. 
The graphs $P(2d+2,d)$ where $d\ge3$ is odd, are all nontrivial cubic totally silver graphs.

\begin{thm}
For all $n\ge4$, there exists a nontrivial cubic totally silver graph of order~$4n$.
\end{thm}

\begin{proof}
We give two families $E_n$ and $M_n$ of cubic graphs of order~$4n$.
The graph $E_n$ is totally silver when $3|n$ and $M_n$ is totally
silver when $3\not|n$.
Given $n\ge3$, the graph $E_n$ consists of three disjoint $n$--cycles
$u_1u_2\cdots u_nu_1$, $v_1v_2\cdots v_n v_1$, and $w_1w_2\cdots w_n w_1$,
and $n$ vertices $z_1,z_2,\ldots,z_n$, where $z_i$ is joined to $u_i$, $v_i$
and $w_i$ for each $i=1,2,\ldots,n$.
The graph $E_6$ is illustrated in Figure~\ref{fig:EM}.
It is straightforward to see that $E_n$ is $3$--connected, and that for $n\ge6$,
$E_n$ has no $3$--, $4$--, or $5$--cycles.
If $3|n$, let $c:V(E_n)\to\{0,1,2,3\}$ be defined by
$c(u_i)=i\mod3$, $c(v_i)=i+1\mod3$, $c(w_i)=i+2\mod3$, and $c(z_i)=3$.
Then $c$ is a totally silver coloring of~$E_n$.

Given $n\ge3$, the graph $M_n$ consists of a $3n$--cycle $u_1u_2\cdots u_{3n}u_1$,
and $n$ vertices $z_1,z_2,\ldots,z_n$,
and edges $z_i u_i$, $z_i u_{n+i}$, and $z_i u_{2n+i}$ for all $i=1,2,\ldots,n$.
The graph $M_4$ is illustrated in Figure~\ref{fig:EM}
It is straightforward to check that $M_n$ is $3$--connected and that for $n\ge4$,
$M_n$ has girth~$6$. If $3\not|n$, we define $c:V(M_n)\to\{0,1,2,3\}$ by
$c(u_i)=i\mod3$ and $c(z_i)=3$. It is easy to verify that $c$ is a totally silver
coloring.
\end{proof}

\begin{figure}[ht]
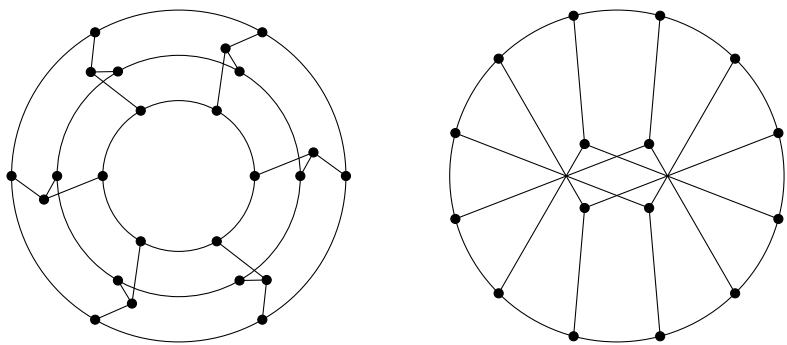

\begin{center}
\begin{tabular}{c@{\hspace{3em}}c}
\includegraphics{tsfig-3.mps}
&
\includegraphics{tsfig-4.mps}
\end{tabular}
\end{center}
\caption{The graphs $E_6$ (left) and $M_4$ (right).}
\label{fig:EM}
\end{figure}

\begin{rem}
Indeed the graph $E_n$ is totally silver if and only if $3|n$. This is since it is
easily observed that in every totally silver coloring of $E_n$, all the vertices
$z_n$ receive the same color. This implies that each of the remaining $n$--cycles
is colored by three colors. On the other hand, a $3$--coloring of $C_n\2$ exists
if and only if $3|n$.
A brute-force computer search confirms that $\chi(E_5\2)=6$. 
For a $5$--coloring of $E_n\2$ where $n>5$, one
may take a $4$--coloring of $C_n\2$ for the cycle on the $u_i$, the same coloring
rotated one position for the cycle on the $v_i$, and rotated two positions for
the cycle on the $w_i$. A fifth color is used for the~$z_i$.
\end{rem}

\begin{rem}
Similarly, it is easily seen that in every totally silver coloring
of $M_n$, all $z_i$ receive the same color, thus only three colors are available
for the remaining $3n$--cycle. On the other hand, the only $3$--coloring of the
square of a $3n$--cycle has color classes which contain every third vertex on the
cycle. If $3|n$, this coloring is not valid as a total silver coloring of $M_n$.
This is since if $3|n$, then $c(x_1)=c(x_n)$, while the vertices $x_1$ and $x_n$
have a common neighbor~$z_1$.
A brute-force computer search confirms that $\chi(M_3\2)=6$. For a $5$--coloring
of $M_n\2$ when $n>3$, one may use four colors for the $3n$--cycle on the $x_i$ and
a fifth color for the $z_i$. We omit a description of the coloring of the cycle
since it involves several cases.
\end{rem}

For $n\ge 4$, the graph $L_{2n}$ consists of a $2n$--cycle $x_1\cdots x_{2n}$
and the edges $x_ix_{i+5}$ for all odd $i$ (subscripts are reduced modulo $2n$). 
In particular, $L_8$ is isomorphic to the hypercube $Q_3$ and $L_{14}$ is isomorphic
to the Heawood graph.
The graph $L_{24}$ is illustrated in Figure~\ref{fig:L24}.
The graph $L_{2n}$ is bipartite for all $n\ge4$, since every edge joins an
odd-numbered vertex to an even-numbered vertex. Thus by Corollary~\ref{cor:2r+2},
$L_{2n}$ is not totally silver when $4\not|n$.

On the other hand, $L_{2n}$ is totally silver when $4|n$.
A totally silver coloring of $L_{2n}$ in this case is more easily presented
using an alternate representation of this graph.
For $m\ge1$, we construct a graph $L'_m$ on $8m$ vertices as follows.
We begin by a $6m$--cycle $y_1y_2\cdots y_{6m}y_1$, and for each $1\le i\le6m$
with $i=1,2\mod 6$, we join a new vertex to the vertices $y_i$, $y_{i+4}$, and $y_{i+8}$
(all subscripts are reduced modulo $6n$). We refer to the $2m$ new vertices added
in this last step as $z_1,z_2,\ldots,z_{2m}$. 
A totally silver coloring of $L'_m$ is the map $c:V(L'_m)\to\{0,1,2,3\}$
defined with $c(y_i)=i\mod3$ and $c(z_i)=3$. It remains to prove that
$L'_m$ is isomorphic to $L_{8m}$. We omit a formal proof of this result,
and instead, we present a graphical proof in Figure~\ref{fig:L24},
where in the normal drawing of the graph $L_{24}$ (left) we highlight an
$18$--cycle corresponding to the main cycle of the graph $L'_3$ (right).
The clear pattern visible in this figure may be used to give an
isomorphism between the graphs $L'_m$ and $L_{8m}$.

\begin{figure}[htb]
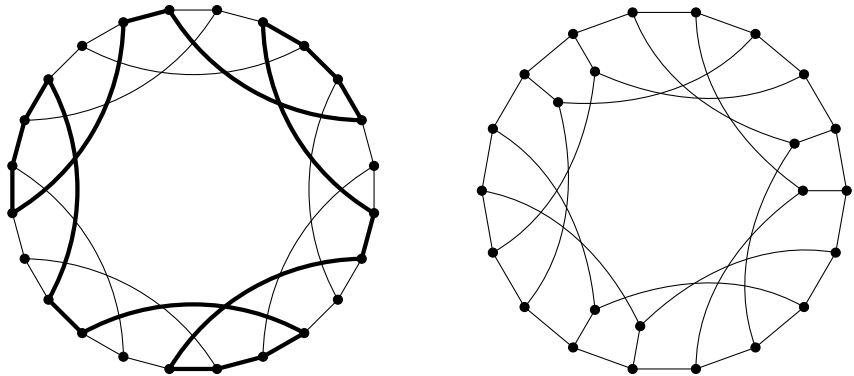

\begin{center}
\begin{tabular}{c@{\hspace{3em}}c}
\includegraphics{tsfig-2.mps}
&
\includegraphics{tsfig-201.mps}
\end{tabular}
\end{center}
\caption{Two drawings of the graph $L_{24}$.}
\label{fig:L24}
\end{figure}

For a positive integer $n\ge 2$, the {\em M\"obius ladder} $V_{2n}$ consists of a cycle
$C=v_1v_2\cdots v_{2n}v_1$ and the chords $v_iv_{n+i}$ for all $i=1,2,\ldots,n$.
Note that $V_4$ is isomorphic to $K_4$ and $V_6$ is isomorphic to $K_{3,3}$.
The graph $V_{2n}$ contains a $4$--cycle $v_1v_2v_{n+2}v_{n+1}v_1$.
Using the reduction of Lemma~\ref{lem:reduce4cycle}, and the fact that $V_4$ is
totally silver, one can see that $V_{2n}$ is totally silver if and only if
$n=2\mod4$. We may use M\"obius ladders in construction of nontrivial cubic totally
silver graphs as follows.
Let $n\ge2$ and let $H$ be obtained from $V_{2n}$ by subdividing all edges of~$C$.
We construct a graph $D_{n}$ from the disjoint union of two copies of $H$
by adding edges between pairs of corresponding vertices of degree $2$ in the
two copies of~$H$. The graph $D_{3}$ is illustrated in Figure~\ref{fig:D3}.

\begin{figure}[htb]
\begin{center}
\includegraphics{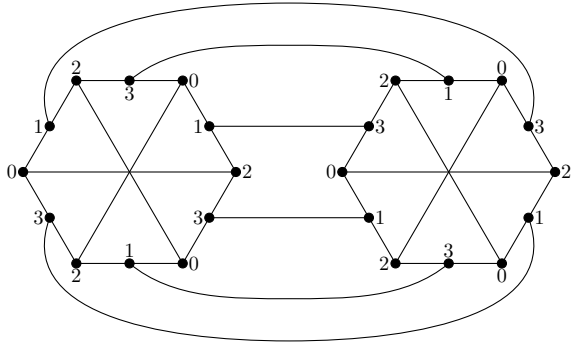}
\end{center}
\caption{The graph $D_3$ and a totally silver coloring of it.}
\label{fig:D3}
\end{figure}

Let $x_1x_2\cdots x_{4n}x_1$ and $y_1y_2\cdots y_{4n}y_1$ be the cycles
in $D_n$ obtained by subdividing the main $2n$--cycle of each copy of
$V_{2n}$ used in its construction. We may shift the labels so that
the edges $x_i y_i$ are present in $D_n$ for all even $1\le i\le4n$.
If $n$ is odd, $c:V(D_n)\to\{0,1,2,3\}$ defined by
$c(x_i) = i\mod 4$ and $c(y_i)=i+2\mod 4$ is a totally silver coloring.
A totally silver coloring of $D_3$ is presented in Figure~\ref{fig:D3}.
Indeed the graph $D_n$ is not totally silver if $n$ is even. We omit the
proof of this assertion.


\section{Totally silver graphs with large girth}
\label{sec:gir}

By way of Corollary~\ref{cor:augment}, any cubic totally silver graph $G$ is obtained from a $2$--regular totally silver graph (a disjoint union of cycles if lengths divisible by~$3$).
Here we give a construction of a family of cubic totally silver graphs with girth $9$ starting from a $3n$--cycle $C=x_1x_2\cdots x_{3n}x_1$ and $n$ isolated vertices $z_1,\ldots,z_n$.
Suppose that $X=\{x_1,x_2,\ldots,x_{3n}\}$ is partitioned into $X=X_1\cup\ldots\cup X_n$, where $|X_i|=3$ for all $i$, and no $X_i$ contains two vertices $x_j$ and $x_k$ where $j$ and $k$ are congruent modulo~$3$. Let $G$ be the cubic graph constructed from the  $3n$--cycle $C$ and the vertices $z_1,\ldots,z_n$, by joining $z_i$ to $X_i$, for all $i=1,2,\ldots,n$. Then $G$ is totally silver. The challenge here is to partition $X$ in such a way that the resulting cubic totally silver graph $G$ contains no short cycles. In the following we show that girth $9$ can be achieved when $n\ge15$, by taking $X_i=\{x_{3i},x_{3i+7},x_{3i+20}\}$ for all $i=1,2,\ldots,n$.

\begin{thm}
For every $n\ge15$, there exists a cubic totally silver graph of order $4n$ and girth~$9$.
 \label{thm:girth9}
\end{thm}

\begin{proof}
Let $n\ge15$ and let $G$ be the graph obtained from a $3n$--cycle $x_0x_1,\ldots,x_{3n-1}x_0$ and $n$ isolated vertices $z_0,z_1,\ldots,z_{n-1}$ according to  the above construction, with the choice of $X_i=\{x_{3i},x_{3i+7},x_{3i+20}\}$ for all $i=0,1,\ldots,n-1$.
To show that $G$ contains no cycles shorter than $9$, by symmetries of the construction of $G$, we may only verify that none of the vertices $x_0$, $x_1$, $x_{3n-1}$, and $z_0$ is contained in a short cycle. Although cumbersome, this can be done by listing all vertices at distance at most $4$ from each of these vertices.
In Figure~\ref{fig:gir9}, we present the subgraph of $G$ induced on the set of vertices at distance at most $4$ from~$x_0$. This illustration affirms that there is no cycle in $G$ of length $8$ or less through the vertex~$x_0$. Similar illustrations for $x_1$, $x_{3n-1}$, and $z_0$ show the same conclusion for these vertices. We omit these, but note that these three subgraphs of $G$ have significant overlap with the one illustrated in Figure~\ref{fig:gir9}.
\end{proof}

\begin{figure}[ht]
\begin{center}
\includegraphics{tsfig-100}
\end{center}
\caption{The subgraph induced by vertices at distance at most $4$ from $x_0$, in the construction of Theorem~\ref{thm:girth9}. Indices of the $x_i$ are reduced modulo $3n$ and indices of the $z_i$ are reduced modulo~$n$. For better visibility, the edges $x_4x_5$, $x_4z_{-1}$, $x_{-4}x_{-5}$, $x_{-5}z_{-4}$, $z_2x_{13}$, and $z_7x_{-22}$ are omitted in this drawing.}
\label{fig:gir9}
\end{figure}

\begin{rem}
It can be verified similarly that the above construction with $n\ge22$ and $X_i=\big\{x_{3i},x_{3i+8},x_{3i+22}\big\}$ for all $i$, gives cubic totally silver cubic graphs of girth~$10$. An alternate way of obtaining (bipartite) totally silver cubic graphs of girth $10$ is discussed in the remainder of this section.
\end{rem}

\begin{rem}
The construction of Theorem~\ref{thm:girth9}, where the neighborhoods of the vertices $z_i$ are congruent modulo~$3$, does not give graphs of girth higher than~$10$. This is since if $z_0$ is adjacent to $x_0$ and $x_p$, then $z_1$ is adjacent to $x_3$ and $x_{p+3}$, thus creating a $10$--cycle $z_0x_0x_1x_2x_3z_1x_{p+3}x_{p+2}x_{p+1}x_pz_0$.
\end{rem}

Given any graph $G$, the {\em bipartite double cover of $G$} is the
tensor product $H=G\times K_2$. Namely, $H$ has vertex set $V(H)=V(G)\times\{1,2\}$,
and $(u,i),(v,j)\in V(H)$ are adjacent in $H$ if and only if
$u v\in E(G)$ and $i\not=j$.
For example $K_{n,n}$ is the bipartite double cover
of~$K_n$. If $c$ is a totally silver coloring of a graph $G$, then the map
$c'$ defined by $c'(u,i)=c(u)$ is a totally silver coloring of~$G\times K_2$.
It is easy to see that every $m$--cycle in $G$ is mapped to a $2m$--cycle in
$G\times K_2$ if $m$ is odd, and to two $m$--cycles in $G\times K_2$ if $m$ is even.
Therefore, the girth of $G\times K_2$ is grater than or equal to the girth of~$G$.
We thus obtain (bipartite) cubic totally silver graphs of girth~$10$ by taking the bipartite double cover of the graphs of Theorem~\ref{thm:girth9}.

We conclude this paper by the following problem.

\begin{pr}
Do there exist cubic totally silver graphs of arbitrary high girth?
\end{pr}

\def\soft#1{\leavevmode\setbox0=\hbox{h}\dimen7=\ht0\advance \dimen7
  by-1ex\relax\if t#1\relax\rlap{\raise.6\dimen7
  \hbox{\kern.3ex\char'47}}#1\relax\else\if T#1\relax
  \rlap{\raise.5\dimen7\hbox{\kern1.3ex\char'47}}#1\relax \else\if
  d#1\relax\rlap{\raise.5\dimen7\hbox{\kern.9ex \char'47}}#1\relax\else\if
  D#1\relax\rlap{\raise.5\dimen7 \hbox{\kern1.4ex\char'47}}#1\relax\else\if
  l#1\relax \rlap{\raise.5\dimen7\hbox{\kern.4ex\char'47}}#1\relax \else\if
  L#1\relax\rlap{\raise.5\dimen7\hbox{\kern.7ex
  \char'47}}#1\relax\else\message{accent \string\soft \space #1 not
  defined!}#1\relax\fi\fi\fi\fi\fi\fi}


\begin{thebibliography}{1}

\bibitem{perfectcol}
E.~M. Bakker, J.~{}van Leeuwen, and R.~B. Tan.
\newblock Perfect colorings.
\newblock Technical Report RUU-CS-90-35, Department of Computer Science,
  Utrecht University, 1990.

\bibitem{cranston}
D.~W. Cranston and S.-J. Kim.
\newblock List-coloring the square of a subcubic graph.
\newblock {\em Journal of Graph Theory}, 57:65--87, 2008.

\bibitem{EbrJahMah}
J.~Ebrahimi, N.~Jahanbakht, and E.~S. Mahmoodian.
\newblock Vertex domination of generalized petersen graphs.
\newblock {\em Discrete Mathematics}, 309(13):4355--4361, 2009.

\bibitem{silvercubes}
M.~Ghebleh, L.~A. Goddyn, E.~S. Mahmoodian, and M.~Verdian-Rizi.
\newblock {Silver cubes.}
\newblock {\em Graphs Comb.}, 24(5):429--442, 2008.

\bibitem{strongcol}
G.~Kant and J.~{}van Leeuwen.
\newblock Strong colorings of graphs.
\newblock Technical Report RUU-CS-90-16, Department of Computer Science,
  Utrecht University, 1990.

\bibitem{IMO97}
M.~Mahdian and E.~S. Mahmoodian.
\newblock The roots of an {IMO}97 problem.
\newblock {\em Bull. Inst. Combin. Appl.}, 28:48--54, 2000.

\end{thebibliography}
\end{document}